\documentclass[11pt,a4paper]{amsart}
\usepackage{lscape}
\usepackage{hyperref} 
\usepackage{multirow}
\usepackage{amsmath,amsthm, amscd, amssymb, amsfonts}
\usepackage{epsfig}
\usepackage{amsmath,amsthm,amssymb, amscd,enumerate}

\usepackage{color}

\numberwithin{equation}{section}
\theoremstyle{plain}
\newtheorem{theorem} {Theorem} [section]
\newtheorem*{theoremSN} {Theorem}

\newtheorem{corollary} [theorem] {Corollary}
\newtheorem{proposition} [theorem] {Proposition}

\theoremstyle{definition}

\newtheorem{definition}[theorem]{Definition}
\newtheorem{remark}[theorem]{Remark}

\renewcommand \parallel {/\kern-3pt/}
\newcommand \N {\mathbb N}
\newcommand \n {\mathfrak{n}}
\renewcommand \k {\mathbb{K}}
\newcommand \End {\operatorname{End}}

\newcommand \im {\operatorname{Im}}

\begin{document}


\newpage

\title[Minimal faith. repns of the free $2$-step nilp. Lie algebra]{Minimal faithful representations of the free 2-step nilpotent Lie algebra of the rank $r$}


\author[Cagliero]{Leandro Cagliero$^\dagger$}
\address{\noindent $^\dagger$FaMAF-CIEM (CONICET), Universidad Nacional de C\'ordoba,
Medina A\-llen\-de s/n, Ciudad Universitaria, 5000 C\' ordoba, Rep\'
ublica Argentina.}
\email{cagliero@famaf.unc.edu.ar}


\author[Rojas]{Nadina Rojas$^\ddagger$}
\thanks{Partially supported by \textsc{Conicet} PIP 112-2013-01-00511CO, \textsc{SeCyT-UNC} 33620180100983CB, MinCyT Córdoba, FONCYT Pict2013 1391, \textsc{FaCEFyN UNC}}
\address{\noindent $^\ddagger$FaCEFyN-CIEM (CONICET), Universidad Nacional de C\'ordoba,
Medina A\-llen\-de s/n, Ciudad Universitaria, 5000 C\' ordoba, Rep\'
ublica Argentina.}
\email{nadina.rojas@unc.edu.ar}




\begin{abstract}
Given a finite dimensional Lie algebra $\mathfrak{g}$, let
$\mathfrak{z}(\mathfrak{g})$ denote the center of
$\mathfrak{g}$ and let $\mu(\mathfrak{g})$
be the minimal possible dimension for a
faithful representation of $\mathfrak{g}$.
In this paper we obtain $\mu(\mathcal{L}_{r,2})$, where
 $\mathcal{L}_{r,k}$ is the free
$k$-step nilpotent Lie algebra of rank $r$.
In particular we prove that $\mu(\mathcal{L}_{r,2})=
\left\lceil \sqrt{2r(r-1)}  \right\rceil + 2$
for $r \geq 4$.
It turns out that $\mu(\mathcal{L}_{r,2})
\sim\mu\big(\mathfrak{z}(\mathcal{L}_{r,2})\big)
\sim 2\sqrt{\dim\mathcal{L}_{r,2}} $ (as $r\to\infty$) and we present
some evidence that this could be true for  $\mathcal{L}_{r,k}$
for any $k$, this is considerably lower than the known bounds
for $\mu(\mathcal{L}_{r,k})$, which are (for fixed $k$) polynomial
in $\dim\mathcal{L}_{r,k}$.
\end{abstract}

\maketitle


\section{Introduction and main results}\label{intro}


We fix throughout a field $\mathbb{K}$ of characteristic zero
and all vector spaces considered in this paper are assumed to be finite
dimensional over $\mathbb{K}$.

Ado's Theorem states that for any Lie algebra $\mathfrak{g}$
there exists a faithful (finite dimensional) representation of
$\mathfrak{g}$.
Even though there are different proofs of Ado's Theorem
(see for instance \cite{Bu1,dG,KB,Z}), they  do not usually
yield
faithful representations having low dimension compared to
$\dim\mathfrak{g}$, and obtaining efficient algorithms
for producing faithful representations of small dimension
is an active field of research, see for instance \cite{BEdG, dG}.

The situation is similar for polycyclic groups, and in particular for
finitely generated torsion-free nilpotent groups ($\tau$-groups): the works of Auslander \cite{A}, and Jennings for $\tau$-groups \cite{J},  show that these groups can be embedded into
some group of matrices over the integers, but as
in the case of Lie algebras, it is difficult to provide embeddings
of low dimension (compared to the Hirsch length of the given group).
Obtaining algorithms towards this end is an
important problem (see \cite{dGN,LO, N}) and, in fact, the interest in low dimensional faithful representations also applies
to other type of groups and algebras (see for instance \cite{Ba, LC, W}).
This problem for $\tau$-groups is very closely related to that
for nilpotent
Lie algebras, mainly by the exp and log maps,
and many ideas are borrowed from
each other (see \cite{dGN}).
A classical reference for this is the book of Segal \cite{S}.

The interest in low dimensional faithful representations has many
motivations. For instance, there are very few classes of groups for which the isomorphism problem is known to be solvable
and it is acknowledged as a remarkable case the solution obtained by
Grunewald and Segal \cite{GS} for the class of $\tau$-groups, which depend of having faithful representations of the groups.
It is clear here the importance of having algorithms
that provide low dimensional faithful representations.
For Lie algebras, in addition to the connection already mentioned with
groups, Milnor and Auslander related the problem of determining
whether a given group is the fundamental group of a compact complete affinely-flat manifold with that of finding Lie algebras $\mathfrak{g}$ admitting
faithful representations of dimension less than or equal to
$\dim\mathfrak{g}$. These Lie algebras, in turn, yield Lie groups admitting a left-invariant affine structures. Many details about this can be found in \cite{Bu}.

This leads to consider, for a given Lie algebra $\mathfrak{g}$, the invariant
\[
\mu(\mathfrak{g})= \min \{ \dim V : (\pi, V) \text{ is a faithful representation of } \mathfrak{g}\}
\]
which is, in general, very difficult to compute.
Only for a few
families of Lie algebra the value of $\mu$ is known, for example:
semisimple \cite{BM1},
nilpotent of dimension less than or equal to 6 \cite{Ro2},
direct sum of abelian plus Heisenberg Lie algebras \cite{Ro1, S},
current Heisenberg Lie algebras \cite{CR1}.

While it is known  that
$\mu(\mathfrak{g})\le \binom{c+\dim\mathfrak{g}}{c}$ for a $c$-step nilpotent Lie algebra $\mathfrak{g}$ (see \cite{dG}), as far as we know all the evidence indicates
that $\mu(\mathfrak{g})\le K\dim \mathfrak{g}$ for some constant $K$:
to the best of our knowledge, it is not known whether
there is a family of Lie algebras
$\mathfrak{g}_n$ such that $\dim\mathfrak{g}_n\to\infty$ and
$\mu(\mathfrak{g}_n)=O(\dim\mathfrak{g}^c)$, $c>1$;
and
it is not known whether $\mu(\mathfrak{g})$ is bounded above by a polynomial in $\dim\mathfrak{g}$.

\medskip

In this paper we consider the free
$2$-step nilpotent Lie algebra of rank $r$
$\mathcal{L}_{r,2}=\mathbb{K}^r\oplus \bigwedge^2 \k^r$ and we prove the following theorem.

\begin{theoremSN}
Let
$\mathcal{L}_{r,2}$ be the free $2$-step nilpotent Lie algebra of rank
$r$. Then
\[
 \mu(\mathcal{L}_{r,2}) =
 \begin{cases}
  \left\lceil \sqrt{2r(r-1)}  \right\rceil + 2,
  & \text{if $r \geq 4$;} \\[4mm]
  2r-1,
  & \text{if $r=2, 3, 4, 5$.}
 \end{cases}
\]
\end{theoremSN}
That is
\begin{center}
\begin{tabular}{rccccccccccccccc}
 $r$: & 2 & 3 & 4 & 5 & 6 & 7 & 8 & 9 & 10 & 11 & 12 & 13 & 14 & 15 & 16  \\[1mm]
 $\mu(\mathcal{L}_{r,2})$: & 3 & 5 & 7 & 9 & 10 & 12 & 13 & 14 & 16 & 17 & 19 & 20 & 22 & 23 & 24
\end{tabular}
\end{center}

\medskip

This theorem says that
\begin{align*}
\mu(\mathcal{L}_{r,2})&\sim \sqrt{2}\,r \\
&\sim 2\sqrt{\dim\mathcal{L}_{r,2}} \\
& \sim \mu(\mathfrak{z}(\mathcal{L}_{r,2})),
\end{align*}
where $\mathfrak{z}(\mathcal{L}_{r,2})=\bigwedge^2 \k^r$ is the center of $\mathcal{L}_{r,2}$.

The value $\mu(\mathcal{L}_{r,2})$ is surprisingly low for us, note that it follows at once that
$\mu(\mathcal{L}_{r,2})\ge \mu(\mathfrak{z}(\mathcal{L}_{r,2}))=
\left\lceil\sqrt{2r(r-1)-1}\right\rceil$
(any faithful representation of an abelian Lie algebra of dimension $n$ has dimension greater than or equal to $\lceil2\sqrt{n-1}\rceil$, see \cite{S}).
Consequently, it was a very hard task for us proving that there actually existed
faithful representations having the necessary dimension.

\subsection{Some words about the proof and some evidence about higher nilpotency degrees}
It is natural to look for faithful matrix representations
\[
 \pi:\mathcal{L}_{r,2}\to\mathfrak{gl}(a+p+b,\k),
\]
for some $a,p,b\in\N$, so that
\begin{equation}\label{eq.cond1}
\footnotesize
  \setlength{\unitlength}{8pt}
\begin{picture}(14,10)(2,-1)
\thicklines
\put(-0.5,4){\normalsize$\pi(\mathcal{L}_{r,2})\subset$}
\put(5.5,0){\line(0,1){8}}
\put(15.5,0){\line(0,1){8}}
\put(15.3,8){\line(1,0){.2}}
\put(5.5,8){\line(1,0){.2}}
\put(5.5,0){\line(1,0){.2}}
\put(15.3,0){\line(1,0){.2}}

\linethickness{0.1mm}
\multiput(11.5,0)(0,.2){40}{\line(0,1){.1}}
\multiput(9.5,0)(0,.2){40}{\line(0,1){.1}}
\multiput(5.5,3)(.2,0){50}{\line(1,0){.1}}
\multiput(5.5,5)(.2,0){50}{\line(1,0){.1}}

\put(13.2,3.7){$*$}
\put(13.2,6.35){$*$}
\put(10.2,6.35){$*$}
\put(7.2,6.35){$ 0 $}
\put(7.2,3.7){$0$}
\put(10.2,3.7){$0$}
\put(7.2,1.2){$ 0 $}
\put(10.2,1.2){$0$}
\put(13.2,1.2){$0$}
\put(11.85,0){$\underbrace{\rule{28pt}{0pt}}_\text{\small{$b$}}$}
\put(9.6,0){$\underbrace{\rule{10pt}{0pt}}_\text{\small{$p$}}$}
\put(5.7,0){$\underbrace{\rule{30pt}{0pt}}_\text{\small{$a$}}$}

\put(15.35,6.25){$\left.\rule{0mm}{4mm}\right\}\text{{$a$}}$}
\put(15.35,3.75){$\left.\rule{0mm}{2mm}\right\}\text{{$p$}}\;\;,$}
\put(15.35,1.25){$\left.\rule{0mm}{4mm}\right\}\text{{$b$}}$}
\end{picture}
\qquad\qquad\qquad
\begin{picture}(14,10)(2,-1)
\thicklines
\put(-1.9,4){\normalsize$\pi( \bigwedge^2 \k^r)\subset$}
\put(5.5,0){\line(0,1){8}}
\put(15.5,0){\line(0,1){8}}
\put(15.3,8){\line(1,0){.2}}
\put(5.5,8){\line(1,0){.2}}
\put(5.5,0){\line(1,0){.2}}
\put(15.3,0){\line(1,0){.2}}

\linethickness{0.1mm}
\multiput(11.5,0)(0,.2){40}{\line(0,1){.1}}
\multiput(9.5,0)(0,.2){40}{\line(0,1){.1}}
\multiput(5.5,3)(.2,0){50}{\line(1,0){.1}}
\multiput(5.5,5)(.2,0){50}{\line(1,0){.1}}

\put(13.2,3.7){$0$}
\put(13.2,6.35){$*$}
\put(10.2,6.35){$0$}
\put(7.2,6.35){$ 0 $}
\put(7.2,3.7){$0$}
\put(10.2,3.7){$0$}
\put(7.2,1.2){$ 0 $}
\put(10.2,1.2){$0$}
\put(13.2,1.2){$0$}
\put(11.85,0){$\underbrace{\rule{28pt}{0pt}}_\text{\small{$b$}}$}
\put(9.6,0){$\underbrace{\rule{10pt}{0pt}}_\text{\small{$p$}}$}
\put(5.7,0){$\underbrace{\rule{30pt}{0pt}}_\text{\small{$a$}}$}

\put(15.35,6.25){$\left.\rule{0mm}{4mm}\right\}\text{{$a$}}$}
\put(15.35,3.75){$\left.\rule{0mm}{2mm}\right\}\text{{$p$}}\;.$}
\put(15.35,1.25){$\left.\rule{0mm}{4mm}\right\}\text{{$b$}}$}
\end{picture}
\end{equation}

\medskip

\noindent
In this case, it is necessary that $ab\ge r(r-1)/2= \dim  \bigwedge^2 \k^r$.
We prove that given  $a,b\in\N$ such that $a+b$ takes the minimal possible value subject to $ab\ge r(r-1)/2$,
such representation $\pi$ is possible for $p=2$ (and impossible for $p=1$).
In fact, it turns out that given $a,b\in\N$, with
$ab\ge r(r-1)/2$,
then any random injective map
$\pi_0:\k^r\to \text{Hom}(\mathbb{K}^2,\mathbb{K}^a)\oplus
\text{Hom}(\mathbb{K}^b,\mathbb{K}^2)$
extends to a faithful representation
$\pi:\mathcal{L}_{r,2}\to\text{End}(\mathbb{K}^a\oplus\mathbb{K}^2\oplus\mathbb{K}^c)$
as in \eqref{eq.cond1}.

We think that this random
property of minimal representations could provide
a new perspective for constructing low-dimensional representations.
In particular,
it seems to us that this is a general pattern for the
free  $k$-step nilpotent Lie algebra on $r$ generators $\mathcal{L}_{r,k}$.
More precisely, we think that the following claim is true:
given $a_0,a_k\in\N$ such that $a_0a_k\ge \dim\mathfrak{z}(\mathcal{L}_{r,k})$,
then there are $a_1,\dots,a_{k-1}\in\N$,
very low compared to $\max\{a_0,a_k\}$,
such that any random injective map
$\pi_0:\k^r\to \bigoplus_{i=1}^k\text{Hom}(\mathbb{K}^{a_i},\mathbb{K}^{a_{i-1}})$
extends to a faithful representation
$\pi:\mathcal{L}_{r,k}\to\text{End}\big(\bigoplus_{i=0}^k\mathbb{K}^{a_i}\big)$.
Here, `very low compared to' means that, eventually,
$\mu(\mathcal{L}_{r,k})
\sim \mu(\mathfrak{z}(\mathcal{L}_{r,k}))\sim 2\sqrt{\frac{r^k}{k}}$.
This is considerably lower than the upper bound $\mu(\n)\le (\dim\n)^k+1$
for a $k$-step nilpotent Lie algebra $\n$ (see \cite{Re}).
Note also that the Lo and Ostheimer algorithm produces a representation of
dimension $1 +r+r^2+\dots+r^k$ for the free nilpotent group of rank $r$
and class $k$ (see \cite[Prop. 6.1]{LO}).
For low ranks, our computer experiments show that

\medskip

\begin{center}
\begin{tabular}{rccccccccc }
 $r=$ & 2 & 3 & 4 & 5 & 6 & 7 & 8 & 9 & 10   \\[3mm]
 $\mu(\mathcal{L}_{r,3})\le $ & 6 & 9 & 14 & 18 & 22 & 27 & 32 & 37 & 43  \\[2mm]
$\frac{\text{2nd row}}{2\sqrt{r^3/3}}=$ & 1.84 & 1.50 & 1.52 & 1.40 & 1.30 & 1.26 & 1.23 & 1.18 & 1.18  \\[5mm]
 $\mu(\mathcal{L}_{r,4})\le $ & 8 & 15 & 23 & 34 & 47 & 62 & 79 & 101 & 122\\[2mm]
  $\frac{\text{4th row}}{2\sqrt{r^4/4}}=$ & 2.00 & 1.67 & 1.44 & 1.36 & 1.30 & 1.26 & 1.24 & 1.24 & 1.22
\end{tabular}
\end{center}

\vspace{1mm}

The paper is organized as follows.
In \S2 we give some basic results about $\mathcal{L}_{r,2}$.
In \S3 we prove the lower bound for $\mu(\mathcal{L}_{r,2})$
for all $r \in \mathbb{N}$. Our proof is basically by induction and requires some previous technical results. Even though the lower bound
is so close to $\mu(\mathfrak{z}(\mathcal{L}_{r,2}))$
our proof turns out to be laborious and even the basic case $r=5$ in the induction argument requires a considerable amount of work.
Finally, in \S4 we prove that the proposed value for
$\mu(\mathcal{L}_{r,2})$ is actually attained by a difficult existence argument: we would be very interested in an explicit map describing a representation of $\mathcal{L}_{r,2}$ of dimension
 $\mu(\mathcal{L}_{r,2})$.


\section{Preliminaries}\label{preliminares}


Let $\mathfrak{g}$ be a  Lie algebra and let $V$ be a vector space.
A representation $(\pi, V)$ of $\mathfrak{g}$ on $V$ is a Lie
homomorphism $\pi: \mathfrak{g} \rightarrow \mathfrak{gl}(V)$
and we say that
\begin{enumerate}[(1)]
\item $(\pi, V)$ is \emph{faithful} if $\pi$ is an injective.
\item $(\pi, V)$ is a \emph{nilrepresentation} if $\pi(X)$ is a nilpotent endomorphism for all
      $X \in \mathfrak{g}$.
\end{enumerate}
Let
$$
\mu_{nil}(\mathfrak{g})= \min\{\dim V : (\pi, V) \text{ is a faithful
nilrepresentation of } \mathfrak{g}\}.
$$
We know that if $\mathfrak{g}$ is nilpotent and
the center $\mathfrak{z}(\mathfrak{g})$
of $\mathfrak{g}$ is contained in $[\mathfrak{g}, \mathfrak{g}]$,
then
\begin{equation}\label{eq:munil mu}
\mu_{nil}(\mathfrak{g})= \mu(\mathfrak{g})
\end{equation}
see \cite[Theorem 2.4]{CR1}.

Given $r \in \mathbb{N}$, the free $2$-step nilpotent Lie algebra of rank $r$
 is the vector space
 \[
\mathcal{L}_{r,2}= \k^r\oplus \bigwedge^2 \k^r
 \]
 equipped with the Lie algebra structure
 \[
  [X,Y]= X\wedge Y,\quad X,Y\in\k^r.
 \]
For example, $\mathcal{L}_{2,2}$ is the Heisenberg Lie algebra of dimension $3$.

The Lie algebra $\mathcal{L}_{r,2}$ has dimension $r + \frac{r(r-1)}{2}$ and possesses the following universal property: if  $\mathfrak{h}$ is a Lie algebra and $f: \k^r\to \mathfrak{h}$ is a linear map satisfying
\[
[f(\k^r), [f(\k^r), f(\k^r)]] = 0,
\]
then there is a unique extension
$\bar f : \mathcal{L}_{r,2} \to \mathfrak{h}$ of $f$
to a homomorphism of Lie algebras.

Since $\mathfrak{z}(\mathcal{L}_{r,2})=[\mathcal{L}_{r,2},  \mathcal{L}_{r,2}]$,
 it follows from  \eqref{eq:munil mu} that
$ \mu_{nil}(\mathcal{L}_{r,2})= \mu(\mathcal{L}_{r,2})$.
Some additional properties of  $\mathcal{L}_{r,2}$ are
stated in the following proposition.

\begin{proposition}\label{prop.basic}
If $\mathfrak{a}$ is a proper
Lie subalgebra of $\mathcal{L}_{r,2}$ then
$\mathfrak{a}=\mathfrak{a}_1\oplus\mathfrak{a}_2$ with
$\mathfrak{a}_1$ and $\mathfrak{a}_2$ subalgebras,
$\mathfrak{a}_2\subset\bigwedge^2 \k^r$, such that
$\mathfrak{a}_1$ is either zero, 1-dimensional or
isomorphic to a free 2-step nilpotent Lie algebra
of rank less than $r$.
In particular, if the center of $\mathfrak{a}$ is not contained in $\bigwedge^2 \k^r$, then
$\mathfrak{a}$ is abelian and
$\dim\big(\mathfrak{a}/(\mathfrak{a} \cap \bigwedge^2 \k^r)\big)= 1$.
\end{proposition}

\begin{proof}
 Let
 $\{A_1,\dots,A_k\}$ be a basis of a linear complement of
 $\mathfrak{a}\cap\bigwedge^2 \k^r$ in $\mathfrak{a}$.
 If $\mathfrak{a}_1$ is the Lie subalgebra of $\mathfrak{a}$
 generated by $\{A_1,\dots,A_k\}$,
 then $\mathfrak{a}_1\simeq \mathcal{L}_{k,2}$ and, since
 $\mathfrak{a}$ is proper, $k<r$.
 If  $\mathfrak{a}_2$ is a linear complement of $\mathfrak{a}_1$ in $\mathfrak{a}$
 we have $\mathfrak{a}=\mathfrak{a}_1\oplus\mathfrak{a}_2$ and
 $\mathfrak{a}_2\subset\bigwedge^2 \k^r$.
 If the center of $\mathfrak{a}$ is not contained in
 $\bigwedge^2 \k^r$, then $k=1$, for if $k\ge2$ (or $k=0$)
 then the center of
 $\mathfrak{a}$ is $[\mathfrak{a}_1,\mathfrak{a}_1]\cap \mathfrak{a}_2\subset\bigwedge^2 \k^r$.
\end{proof}

 We know that $\mathcal{L}_{r,2}$ is the nilradical of a parabolic subalgebra of the semisimple
Lie algebra of rank $r$ of type $B$, where $\k^r$
corresponds to the set of short positive roots $\epsilon_i$,
$i=1,\dots,r$, and $\bigwedge^2 \k^r$
corresponds to the set
of positive roots $\epsilon_i+\epsilon_j$.
This provides a standard faithful representation
$\pi_0:\mathcal{L}_{r,2}\to \mathfrak{gl}(2r+1,\mathbb{K})$
of dimension $2r+1$.
More precisely,
if $\{X_1, \dots, X_r\}$ is a basis of $\k^r$ and,
for $1 \leq i < j \leq r$, $Z_{ij}=[X_i,X_j]$, then
\newcommand{\qc}{7.5}
\newcommand{\qp}{7.9}
\[
\setlength{\unitlength}{12pt}
\begin{picture}(14,10.15)(-4,-.2)
\thicklines
\put(-10.3,5){$\displaystyle\pi_0\left(\sum_{i= 1}^{r} x_i X_i + \sum_{i<j}^r z_{ij} Z_{ij}\right)=$}
\put(2,0){\line(0,1){10}}
\put(16,0){\line(0,1){10}}
\put(2,0){\line(1,0){.2}}
\put(2,10){\line(1,0){.2}}
\put(15.8,0){\line(1,0){.2}}
\put(15.8,10){\line(1,0){.2}}

\linethickness{0.1mm}
\multiput(\qc,0)(0,.2){50}{\line(0,1){.1}}
\multiput(9,0)(0,.2){50}{\line(0,1){.1}}
\multiput(2,5.5)(.2,0){70}{\line(1,0){.1}}
\multiput(2,4)(.2,0){70}{\line(1,0){.1}}

\put(4.75,4.5){\small$ 0$}
\put(8,4.5){\tiny$ 0$}
\put(12.5,1.5){\small$ 0 $}
\put(4.75,1.5){\small$0$}
\put(8,1.5){\small$ 0 $}
\put(4.75,7.5){\small$0$}
\put(\qp,9.35){\tiny$x_1$}
\put(\qp,8.35){\tiny$x_2$}
\put(8.1,7){\small$\vdots$}
\put(8,6){\tiny$x_r$}
\put(9.3,4.5){\tiny$x_r$}
\put(11.15,4.5){\small$\dots$}
\put(13.5,4.5){\tiny$x_2$}
\put(15,4.5){\tiny$x_1$}
\put(14.85,9.3){\tiny$0$}
\put(13.85,8.35){\tiny$-z_{12}$}
\put(15,7){\small$\vdots$}
\put(13.85,6){\tiny$-z_{1r}$}
\put(12.15,9.3){\tiny$z_{12}$}
\put(12.5,8.2){\tiny$0$}
\put(12.65,7){\small$\vdots$}
\put(11.85,6){\tiny$-z_{2r}$}
\put(10.75,9.3){\small$\dots$}
\put(10.75,7.0){\reflectbox{$\ddots$}}
\put(9.5,6){\tiny$0$}
\put(9.6,7){\small$\vdots$}
\put(9.2,9.3){\tiny$z_{1r}$}
\put(9.2,8.3){\tiny$z_{12}$}
\put(10.75,8.3){\small$\dots$}
\put(10.5,6.1){\small$\dots$}
\put(9.2,-0.05){$\underbrace{\rule{82pt}{0pt}}_\text{\tiny{$r$}}$}
\put(7.4,-0.05){$\underbrace{\rule{10pt}{0pt}}_\text{\tiny{$1$}}$}
\put(1.9,-0.05){$\underbrace{\rule{63pt}{0pt}}_\text{\tiny{$r$}}$}
\put(15.9,7.7){$\left.\rule{0mm}{11mm}\right\}\text{\tiny{$r$}}$}
\put(15.9,4.65){$\left.\rule{0mm}{4mm}\right\}\text{\tiny{$1$}}$}
\put(15.9,1.75){$\left.\rule{0mm}{9mm}\right\}\text{\tiny{$r$}}$}
\end{picture}
\vspace{3.75mm}
\]
\noindent
This shows that $\mu(\mathcal{L}_{r,2}) \leq 2r + 1$.
\begin{definition}\label{def1}
Let $a, p, b$ be natural numbers. We say that a representation of
$\pi:\mathcal{L}_{r,2}\to \mathfrak{gl}(V)$ is
of type $(a,p,b)$ if there is a basis
$$B=\{u_1,\dots,u_a,v_1,\dots,v_p,w_1,\dots,w_b\}$$
of $V$ such that the corresponding matrix representation $\pi_B$
associated to $\pi$ satisfies
\begin{equation}\label{eq.apb_block}
  \setlength{\unitlength}{11pt}
\begin{picture}(14,8.35)(2,-1)
\thicklines
\put(0.5,4){$\pi_B(X)=$}
\put(18,4){ for all $X\in\mathcal{L}_{r,2}$.}
\put(5.5,0){\line(0,1){8}}
\put(15.5,0){\line(0,1){8}}
\put(15.3,8){\line(1,0){.2}}
\put(5.5,8){\line(1,0){.2}}
\put(5.5,0){\line(1,0){.2}}
\put(15.3,0){\line(1,0){.2}}

\linethickness{0.1mm}
\multiput(11.5,0)(0,.2){40}{\line(0,1){.1}}
\multiput(9.5,0)(0,.2){40}{\line(0,1){.1}}
\multiput(5.5,3)(.2,0){50}{\line(1,0){.1}}
\multiput(5.5,5)(.2,0){50}{\line(1,0){.1}}

\put(13.2,3.7){$*$}
\put(13.2,6.35){$*$}
\put(10.2,6.35){$*$}
\put(7.2,6.35){$ 0 $}
\put(7.2,3.7){$0$}
\put(10.2,3.7){$0$}
\put(7.2,1.2){$ 0 $}
\put(10.2,1.2){$0$}
\put(13.2,1.2){$0$}
\put(11.65,0){$\underbrace{\rule{40pt}{0pt}}_\text{\small{$b$}}$}
\put(9.5,0){$\underbrace{\rule{15pt}{0pt}}_\text{\small{$p$}}$}
\put(5.5,0){$\underbrace{\rule{40pt}{0pt}}_\text{\small{$a$}}$}

\put(15.35,6.25){$\left.\rule{0mm}{7mm}\right\}\text{{$a$}}$}
\put(15.35,3.75){$\left.\rule{0mm}{4mm}\right\}\text{{$p$}}$}
\put(15.35,1.25){$\left.\rule{0mm}{7mm}\right\}\text{{$b$}}$}
\end{picture}
\end{equation}
\end{definition}

\begin{remark}
The standard representation $\pi_0$
is a faithful representation of type $(r,1,r)$,
but it is not
of least dimension among all faithful
representations of type $(a,1,b)$.
Indeed, let $\pi_1$ be the representation defined, for $X=\sum_{i= 1}^{r} x_i X_i + \sum_{i<j}^r z_{ij} Z_{ij}$, by
$$
\setlength{\unitlength}{13pt}
\begin{picture}(12,12)(3,-1)
\thicklines
\put(1,-2){\line(0,1){13}}
\put(19.25,-2){\line(0,1){13}}
\put(1,-2){\line(1,0){.2}}
\put(1,11){\line(1,0){.2}}
\put(19.05,-2){\line(1,0){.2}}
\put(19.05,11){\line(1,0){.2}}

\linethickness{0.1mm}
\multiput(7.25,-2)(0,.2){65}{\line(0,1){.1}}
\multiput(9.25,-2)(0,.2){65}{\line(0,1){.1}}
\multiput(1,5.5)(.2,0){92}{\line(1,0){.1}}
\multiput(1,4)(.2,0){92}{\line(1,0){.1}}

\put(4,4.5){\small$0$}
\put(8.15,4.5){\tiny$ 0$}
\put(13.75,0.75){\small$0$}
\put(4,0.75){\small$0$}
\put(8.25,0.75){\small$0$}
\put(4,8){\small$0$}
\put(8,10.5){\tiny$x_1$}
\put(8,9.5){\tiny$x_2$}
\put(8,8.5){\tiny$x_3$}
\put(8.15,7.5){\tiny$\vdots$}
\put(7.75,6.85){\tiny$x_{r-2}$}
\put(7.75,5.85){\tiny$x_{r-1}$}
\put(10.5,4.5){\tiny$x_{r}$}
\put(12,4.5){\tiny$x_{r-1}$}
\put(13.75,4.5){\tiny$\dots$}
\put(15.5,4.5){\tiny$x_3$}
\put(18,4.5){\tiny$x_2$}
\put(15.25,9.3){\tiny$z_{23}$}
\put(15.25,10.5){\tiny$z_{13}$}
\put(15.5,8.5){\tiny$0$}
\put(15.7,7.5){\tiny$\vdots$}
\put(14.75,6){\tiny$-z_{3r-1}$}
\put(14.75,6.85){\tiny$-z_{3r-2}$}
\put(17.5,10.5){\tiny$z_{12}$}
\put(18,9.2){\tiny$0$}
\put(17.25,8.5){\tiny$-z_{23}$}
\put(18,7.5){\tiny$\vdots$}
\put(17.1,6.85){\tiny$-z_{2r-2}$}
\put(17.1,6){\tiny$-z_{2r-1}$}
\put(13.75,10.5){\tiny$\dots$}
\put(13.75,9.5){\tiny$\dots$}
\put(13.75,8.6){\tiny$\dots$}
\put(14,7.4){\reflectbox{$\ddots$}}
\put(13.75,6.8){\tiny$\dots$}
\put(13.75,6){\tiny$\dots$}
\put(11.5,6.85){\tiny$z_{r-2 r-1}$}
\put(12.25,5.85){\tiny$0$}
\put(9.75,5.85){\tiny$z_{r-1 r}$}
\put(9.75,6.85){\tiny$z_{r-2 r}$}
\put(10,7.5){\tiny$\vdots$}
\put(11.5,10.5){\tiny$z_{1r-1}$}
\put(11.5,9.5){\tiny$z_{2r-1}$}
\put(11.5,8.5){\tiny$z_{3r-1}$}
\put(9.75,10.5){\tiny$z_{1r}$}
\put(9.75,9.5){\tiny$z_{2r}$}
\put(9.75,8.5){\tiny$z_{3r}$}
\put(11.75,7.5){\tiny$\vdots$}
\put(9.27,-2){$\underbrace{\rule{129pt}{0pt}}_\text{\tiny{$r-1$}}$}
\put(7.28,-2){$\underbrace{\rule{25pt}{0pt}}_\text{\tiny{$1$}}$}
\put(1.1,-2){$\underbrace{\rule{78pt}{0pt}}_\text{\tiny{$r-1$}}$}
\put(-2.7,4.5){$\displaystyle\pi_1\left(X\right)=$}
\put(19.05,8){$\left.\rule{0mm}{14mm}\right\}\text{\tiny{$r-1$}}$}
\put(19.05,4.5){$\left.\rule{0mm}{4mm}\right\}\text{\tiny{$1$}}$}
\put(19.05,0.75){$\left.\rule{0mm}{15mm}\right\}\text{\tiny{$r-1$}}$}
\end{picture}
\vspace{6mm}
$$

\noindent
It is not difficult to see that
$(\pi_1, \k^{2r-1})$ is  faithful
of type $(r-1, 1, r-1)$.
We will show next that $2r-1$ is the least
possible dimension for a faithful
representations of $\mathcal{L}_{r,2}$ of type $(a,1,b)$.
\end{remark}

\begin{proposition}\label{prop.type}
Let $a, p, b \in \mathbb{N}$ and let $\pi:\mathcal{L}_{r,2}\to \mathfrak{gl}(V)$ be
 a faithful representation
of type $(a,p,b)$. Then
\[
r\le p\min(a,b)+1\qquad \text{ and }\qquad  \frac{r(r-1)}{2}\le ab.
\]
In particular, if $p=1$ then the minimal possible dimension of $V$
is $2r-1$.
\end{proposition}

\begin{proof}
Let $B$ be a basis of $V$ as in Definition \ref{def1} and let $\pi_B$ be the corresponding matrix representation.
It is clear that for any $Z\in [\mathcal{L}_{r,2},\mathcal{L}_{r,2}]$,
the non-zero entries of
$\pi_B(Z)$ are contained in the upper-right block of \eqref{eq.apb_block}.
Since $\pi$ is faithful and
$\dim [\mathcal{L}_{r,2},\mathcal{L}_{r,2}]=\frac{r(r-1)}{2}$,
it follows
 that $\frac{r(r-1)}{2}\le ab$.

In order to prove that $r\le p\min(a,b)+1$ let us fix
a set $\{X_1, \dots,X_r\}$ of generators of $\mathcal{L}_{r,2}$ and
assume $b\ge a$.
Let
\[
 M=\text{span}_{\k}\{\pi_B(X_1), \dots, \pi_B(X_r)\}.
\]
If $r\ge pa+2$,
after a Gaussian elimination
process, we may find a basis $\{M_1, \dots,M_r\}$ of $M$
such that $M_{r-1}$ and $M_r$ have the following block structure
\[
  \setlength{\unitlength}{11pt}
\begin{picture}(19,10)(2,-1.5)
\thicklines
\put(5.5,0){\line(0,1){8}}
\put(15.5,0){\line(0,1){8}}
\put(15.3,8){\line(1,0){.2}}
\put(5.5,8){\line(1,0){.2}}
\put(5.5,0){\line(1,0){.2}}
\put(15.3,0){\line(1,0){.2}}

\linethickness{0.1mm}
\multiput(11.5,0)(0,.2){40}{\line(0,1){.1}}
\multiput(9.5,0)(0,.2){40}{\line(0,1){.1}}
\multiput(5.5,3)(.2,0){50}{\line(1,0){.1}}
\multiput(5.5,5)(.2,0){50}{\line(1,0){.1}}

\put(13.2,3.7){$*$}
\put(13.2,6.35){$*$}
\put(10.2,6.35){$0$}
\put(7.2,6.35){$ 0 $}
\put(7.2,3.7){$0$}
\put(10.2,3.7){$0$}
\put(7.2,1.2){$ 0 $}
\put(10.2,1.2){$0$}
\put(13.2,1.2){$0$}
\put(11.65,0){$\underbrace{\rule{42pt}{0pt}}_\text{\small{$b$}}$}
\put(9.5,0){$\underbrace{\rule{22pt}{0pt}}_\text{\small{$p$}}$}
\put(5.5,0){$\underbrace{\rule{42pt}{0pt}}_\text{\small{$a$}}$}

\put(15.35,6.25){$\left.\rule{0mm}{7mm}\right\}\text{{$a$}}$}
\put(15.35,3.75){$\left.\rule{0mm}{4mm}\right\}\text{{$p$}}$}
\put(15.35,1.25){$\left.\rule{0mm}{7mm}\right\}\text{{$b$}}$}
\end{picture}
\]
Hence $[M_{n-1},M_n]=0$ and this contradicts the fact that $\pi$ is faithful.
Finally, if $p=1$ then
$\dim V=a+b+1\ge 2\min(a,b)+1\ge 2r-1$.
\end{proof}


\section{The lower bound for \texorpdfstring{$\mu(\mathcal{L}_{r,2})$}{m(Lr)}}\label{lower}

In this section we will prove that
\begin{equation}\label{eq.lower_bound}
 \mu(\mathcal{L}_{r,2}) \geq
 \begin{cases}
  \left\lceil \sqrt{2r(r-1)}  \right\rceil + 2,
  & \text{if $r \geq 4$;} \\[4mm]
  2r-1,
  & \text{if $r=2, 3, 4, 5$.}
 \end{cases}
\end{equation}

\medskip

The cases $r=2,3$ are well known.
On the one hand, $\mathcal{L}_{2,2}$ is a Heisenberg Lie algebra
and $\mu(\mathcal{L}_{2,2})= 3$ follows from \cite{Bu1}.
On the other hand,
$\dim \mathcal{L}_{3,2}= 6$ and $\mathcal{L}_{3,2}$
is isomorphic to the Lie algebra labeled as
$L_{6,26}$ in the paper \cite{Ro2}, and it
follows that $\mu(\mathcal{L}_{3,2})= 5$.

For the rest of this section we assume $r\ge4$.
The main tool used in the proof of \eqref{eq.lower_bound}
is the following result which is a particular instance of  \cite[Theorem 2.3]{CR2}.

\begin{theorem}\label{teo:descomposicion}
Let  $V$ be a vector space and let
$0\ne\mathfrak{n}_2\subset  \mathfrak{n}_1$ be a chain of vector subspaces of $\End(V)$.
Then there exist natural numbers $s_1 \geq s_2>0$,
a linearly independent set
$\{v_1, \dots, v_{s_1}\} \subset V$
and a family of non-zero
subspaces $\mathfrak{n}_{k,j}\subset \End(V)$, $1\le j\le s_k$
and $k= 1, 2$,
such that:
$$
\begin{array}{cccccccccccccc}
\!&\!{\mathfrak{n}}_{1}\!&\!=
\!&\!{\mathfrak{n}}_{{1},1} \!&\!
\oplus\!&\!{\mathfrak{n}}_{{1},2} \!&\! \oplus  \dots  \oplus \!&\! {\mathfrak{n}}_{{1}, s_{2}}\!&\! \oplus  \dots  \oplus \!&\! {\mathfrak{n}}_{{1}, s_{1}}\\
\!&\!    \cup       \!&\! \!&\! \cup             \!&\!                          \!&\!  \cup          \!&\!\!&\!  \cup          \!&\!                             \!&\!     \\
\!&\!{\mathfrak{n}}_{2}\!&\!=
\!&\!{\mathfrak{n}}_{2,1} \!&\! \oplus
\!&\!{\mathfrak{n}}_{2,2} \!&\! \oplus\dots  \oplus \!&\! {\mathfrak{n}}_{2, s_{2}}
\end{array}
$$
and
\begin{enumerate}[(1)]
\item\label{it.1} $A\in \mathfrak{n}_{1,j}$ and $A v_j=0$ implies $A=0$ for
      $j=1, \dots, s_1$.

\smallskip

\item\label{it.2} $\mathfrak{n}_{1,j}v_i= 0$ for
      $1 \leq i < j \leq s_1$.
\smallskip

\item\label{it.3} $\mathfrak{n}_{k,j}V \subseteq \mathfrak{n}_{k,i}v_i$ for $1 \leq i < j \leq s_k$ and $k=1,2$.

\smallskip

\item\label{it.4} If in addition ${\mathfrak{n}}_1$ consists of nilpotent operators and $[\mathfrak{n}_1,\mathfrak{n}_2]=0$
then
      ${\mathfrak{n}}_{{1},1} v_1 \cap \operatorname{span}_{\k}\{v_1, \dots, v_{s_2}\} = 0$.
\end{enumerate}
\end{theorem}

\begin{remark}\label{rmk.dimension}
 It follows from \eqref{it.1} that
 $\dim\mathfrak{n}_{k,j}=\dim(\mathfrak{n}_{k,j}v_j)$
 for  $1 \leq j \leq s_k$ and $k=1,2$.
 This, combined with \eqref{it.3}, implies
 $\dim\mathfrak{n}_{k,j}\le \dim(\mathfrak{n}_{k,i})$
 for  $1 \leq i<j \leq s_k$ and $k=1,2$.
\end{remark}

\begin{remark}\label{rmk.induction}
This theorem is useful for us since
it allows to argue inductively as follows.
In the context of  Theorem \ref{teo:descomposicion},
let us assume that $\mathfrak{n}_{1}$ is a Lie subalgebra of
$\mathfrak{gl}(V)$.
For any $2\le j_0\le s_2$, let
\[
\tilde{\mathfrak{n}}=\bigoplus_{j=j_0}^{s_1} \mathfrak{n}_{1,j} \qquad
\text{ and }\qquad
V'=\text{span}_{\k}\big(
   \mathfrak{n}_{1,1}v_1 \cup \{v_{j_0},\dots,v_{s_2}\}\big).
\]
We claim that
\begin{enumerate}[(i)]
\item $\tilde{\mathfrak{n}}$ is a Lie subalgebra of $\mathfrak{n}_1$
and it preserves $V'$.

\medskip

\item \label{coro ii}
If $\tilde{\mathfrak{n}}$ is nilpotent with center
$\mathfrak{z}(\tilde{\mathfrak{n}})$
contained in $\mathfrak{n}_2$ then
      $\tilde{\mathfrak{n}}$ acts on $V'$ faithfully.
\end{enumerate}

First we prove (i). Since $j_0\ge 2$,
item \eqref{it.3} implies that $V'$ is
preserved by $\tilde{\mathfrak{n}}$.
Let us show that $\tilde{\mathfrak{n}}$ is a Lie subalgebra of $\mathfrak{n}_1$.
Given $X, Y \in \tilde{\mathfrak{n}}$,
 since $\mathfrak{n}_1$ is a Lie subalgebra of
 $\mathfrak{gl}(V)$ we have
\[
[X, Y]= \sum_{i=1}^{s_1} X_i\quad    \text{ with }\quad
X_i \in \mathfrak{n}_{1, i}.
\]
We must show that $X_i=0$ for all $i<j_0$.
Let $j_1$ be the least $i<j_0$ with $X_{j_1}\ne0$ (if there is any).
Then
\begin{align*}
0 & = XY(v_{j_1})-YX(v_{j_1}) \quad\text{ by definition of
$\tilde{\mathfrak{n}}$ and \eqref{it.2}} \\
  &= [X, Y](v_{j_1}) \\
  & = \sum_{i=1}^{s_1} X_i(v_{j_1})    \quad\text{ } X_i  \in \mathfrak{n}_{1,i}\\
  & = X_{j_1}(v_{j_1}) \quad\text{ by \eqref{it.2}}
\end{align*}
which is a contradiction to \eqref{it.1}.

We now prove (ii).
Since $\tilde{\mathfrak{n}}$ is nilpotent it suffices to
see that its center acts faithfully on $V'$.
It follows from item \eqref{it.1} that
$\bigoplus_{j=j_0}^{s_2} \mathfrak{n}_{2,j}$
acts faithfully on $V'$ since
$\{v_{j_0},\dots,v_{s_2}\}\subset V'$.
Therefore, $\mathfrak{z}(\tilde{\mathfrak{n}})
\subset\mathfrak{n}_2\cap \tilde{\mathfrak{n}}$  acts faithfully on $V'$.
This concludes the remark.
\end{remark}

Let us
fix  a faithful nilrepresentation
$(\pi, V)$ of
$\mathcal{L}_{r,2}$. Let
\[
 \mathfrak{n}_1= \pi(\mathcal{L}_{r,2})\qquad\text{ and }\qquad
 \mathfrak{n}_2= \pi(\mathfrak{z}(\mathcal{L}_{r,2})).
\]
It is clear that $\mathfrak{n}_1$
is a nilpotent Lie subalgebra of
$\mathfrak{gl}(V)$ consisting of nilpotent endomorphisms
which is isomorphic to $\mathcal{L}_{r,2}$.
Thus
\[
 \dim \n_1=r+\frac{r(r-1)}2\qquad\text{ and }\qquad \dim\n_2=\frac{r(r-1)}2
\]

Now the subspaces $\n_{k,i}$ obtained by applying
Theorem \ref{teo:descomposicion} to the chain
${\mathfrak{n}}_{2} \subset {\mathfrak{n}}_{1}$ in this
particular case have some additional properties.

\begin{proposition}\label{teoestruc}
Let
$(\pi, V)$, $\mathfrak{n}_1$ and $\mathfrak{n}_2$ as above,
and let
$s_1\ge s_2>0$, $\{v_1, \dots, v_{s_1}\} \subset V$, and
$\mathfrak{n}_{k,j}\subset \End(V)$
($1\le j\le s_k$, $k= 1, 2$),
be the output obtained by
applying
Theorem \ref{teo:descomposicion} to the chain
${\mathfrak{n}}_{2} \subset {\mathfrak{n}}_{1}$.
Then
$$
\dim \left(\mathfrak{n}_{1,2} \oplus \dots \oplus \mathfrak{n}_{1,s_1} \right) \leq \dim \left(\mathfrak{n}_{2,2} \oplus \dots \oplus \mathfrak{n}_{2,s_2} \right) + r-1.
$$
In the case
\[
\dim \left(\mathfrak{n}_{1,2} \oplus \dots \oplus \mathfrak{n}_{1,s_1} \right) = \dim \left(\mathfrak{n}_{2,2} \oplus \dots \oplus \mathfrak{n}_{2,s_2} \right) + r-1,
\]
there is a Lie subalgebra
$\mathfrak{m}\subset\mathcal{L}_{r,2}$, with
$\mathfrak{m}\simeq \mathcal{L}_{r-1,2}$, such that
$(\pi,V)$ contains a faithful subrepresentation of
$\mathfrak{m}$ of type $(\dim \mathfrak{n}_{1,1}-1,1,s_2-1)$.
In particular
\[
 \dim V \geq
 \begin{cases}
  2r-2 &\text{ for } r\geq 6 \text{ and }\\
 2r-1,& \text{for $r=4, 5$,}
 \end{cases}
\]
and thus \eqref{eq.lower_bound} holds for $r=4,5$.
\end{proposition}

\begin{proof}
It is clar that
\begin{align*}
\dim \left(\mathfrak{n}_{1,2} \oplus \dots \oplus \mathfrak{n}_{1,s_1} \right) -
\dim &\left(\mathfrak{n}_{2,2} \oplus \dots \oplus \mathfrak{n}_{2,s_2} \right) \\
&<
\dim \left(\mathfrak{n}_{1,1} \oplus \dots \oplus \mathfrak{n}_{1,s_1} \right) -
\dim \left(\mathfrak{n}_{2,1} \oplus \dots \oplus \mathfrak{n}_{2,s_2} \right) \\
&=r
\end{align*}
and thus we have the first part of the proposition.

We now assume that
\begin{equation}\label{eq.equal_r-1}
 \dim \left(\mathfrak{n}_{1,2} \oplus \dots \oplus \mathfrak{n}_{1,s_1} \right) = \dim \left(\mathfrak{n}_{2,2} \oplus \dots \oplus \mathfrak{n}_{2,s_2} \right) + r-1,
\end{equation}
and in particular $s_1\ge 2$. Also $s_2 \geq 2$, otherwise
$s_2= 1$ and it follows, from Remark \ref{rmk.induction}, that
$\mathfrak{n}_{1,2} \oplus \dots \oplus \mathfrak{n}_{1,s_1}$ is an abelian Lie subalgebra with  $\left(\mathfrak{n}_{1,2} \oplus \dots \oplus \mathfrak{n}_{1,s_1} \right) \cap \mathfrak{n}_2 = 0$.
Since $\mathfrak{n}_{1}\simeq
\mathcal{L}_{r,2}$, it follows (see
Proposition  \ref{prop.basic}) that
$\dim \left(\mathfrak{n}_{1,2} \oplus \dots \oplus \mathfrak{n}_{1,s_1} \right) = 1$
and hence $r=2$ a contradiction (recall that $r\ge 4$).

From now on we assume $s_1, s_2 \geq 2$.
Let $\mathfrak{m}_0$ be a linear complement of $\mathfrak{n}_{2,2} \oplus \dots \oplus \mathfrak{n}_{2,s_2}$ in
$\mathfrak{n}_{1,2} \oplus \dots \oplus \mathfrak{n}_{1,s_1}$, that is
$$
\mathfrak{n}_{1,2} \oplus \dots \oplus \mathfrak{n}_{1,s_1}=
\mathfrak{m}_0 \; \oplus \;
\mathfrak{n}_{2,2} \oplus \dots \oplus \mathfrak{n}_{2,s_2},
$$
and let $\mathfrak{m}$
be the Lie subalgebra of $\mathfrak{n}_1$ generated $\mathfrak{m}_0$.
Since  $\dim \mathfrak{m}_0=r-1$,
Proposition  \ref{prop.basic}) implies $\mathfrak{m}\simeq\mathcal{L}_{r-1,2}$
and, it follows from Theorem \ref{teo:descomposicion}
item \eqref{it.3} that
\begin{equation}\label{eq.mV}
\mathfrak{m}V\subset \mathfrak{n}_{1,1}v_1.
\end{equation}

From Remark \ref{rmk.induction} we know that
$\mathfrak{n}_{1,2} \oplus \dots \oplus \mathfrak{n}_{1,s_1}$ is a Lie algebra,
containing $\mathfrak{m}\simeq\mathcal{L}_{r-1,2}$, acting faithfully on
\[
 V'=\text{span}_{\k}\big(
   \mathfrak{n}_{1,1}v_1 \cup \{v_{2},\dots,v_{s_2}\}\big).
\]
We claim that the faithful representation of
$\mathcal{L}_{r-1,2}$ given by the action of
$\mathfrak{m}$ on $V'$ is of
type $(n_{1,1}-1,1,s_2-1)$.

Let $n_{i,1}=\dim \mathfrak{n}_{i,1}$, $i=1,2$.
It follows from \eqref{eq.equal_r-1} that $n_{1,1}=1+n_{2,1}$.
Let $\{A_1,\dots,A_{n_{2,1}},A_{n_{1,1}}\}$ be a basis of
$\mathfrak{n}_{1,1}$ with
$A_i\in  \mathfrak{n}_{2,1}$ for $i=1,\dots, n_{2,1}$.
Let
\begin{align*}
  B_1 & = \{A_1v_1,\dots,A_{n_{2,1}}v_1\}, \\
  B_2 & = \{A_{n_{1,1}}v_1\}, \\
  B_3 & = \{v_{2},\dots,v_{s_2}\},
 \end{align*}
and let $V'_i=\text{span}_{\k}(B_i)$, $i=1,2,3$.
It follows from Theorem \ref{teo:descomposicion}
item \eqref{it.1} and \eqref{it.4}
that $B_1\cup B_2 \cup B_3$ is linearly independent and thus
\[
V'=V'_1\oplus V'_2\oplus V'_3.
\]
In order to show that
this representation of $\mathcal{L}_{r-1,2}$ is of
type $(n_{1,1}-1,1,s_2-1)$ we must show that,
given $A\in\mathfrak{m}$, we have
\begin{enumerate}[(i)]
 \item
$Av\in \text{span}_{\k}(B_1\cup B_2)$
for all $v\in V'$,
 \item
$AA_{n_{1,1}}v_1\in \text{span}_{\k}(B_1)$, and
\item
$Av=0$
for all $v\in \text{span}_{\k}(B_1)$.
\end{enumerate}

Property  (i) follows from \eqref{eq.mV}.

Property  (ii) is obtained as follows
\begin{align*}
  AA_{n_{1,1}}v_1 & = A_{n_{1,1}}Av_1+[A,A_{n_{1,1}}]v_1 \\
                 & = [A,A_{n_{1,1}}]v_1 \qquad\qquad\text{by Theorem \ref{teo:descomposicion} item \eqref{it.2}}\\
   & \in   \mathfrak{n}_{2}v_1 \\
   & \subset   \mathfrak{n}_{2,1}v_1 =V'_1\qquad\qquad\text{by Theorem \ref{teo:descomposicion} item \eqref{it.2}.}
 \end{align*}
Finally, Theorem \ref{teo:descomposicion} item \eqref{it.2} implies
$AA_{j}v_1=A_{j}Av_1=0$ for $j=1,\dots,n_{2,1}$,  and this proves (iii).

This shows that the action of $\mathfrak{m}$ on $V'$ corresponds
to a faithful representation  of $\mathcal{L}_{r-1,2}$ of type
$(n_{1,1}-1,1,s_2-1)$, and hence, by Proposition \ref{prop.type}, we obtain
that $\dim V'\ge 2r-3$.
But $\dim V> \dim V'$ since
$v_1\not\in V'$ (see Theorem \ref{teo:descomposicion}  item \eqref{it.4}),
and hence $\dim V>2r-2$.
This completes the proof for $r\ge 6$.

Now assume $r= 5$,
in this case,
$\mathfrak{m}\simeq\mathcal{L}_{4,2}$,
we have proved $\dim V\ge 8$,  $\dim V'\ge 7$ and
we must prove that $\dim V\ge 9$.

Assume, if possible that $\dim V= 8$.
This implies
\[
 \dim V=\dim V'+1,\qquad \dim V'=n_{1,1} + s_2 - 1=7.
\]
In particular
$B=B_1\cup B_2 \cup B_3\cup\{v_1\}$
 is a basis of $V$.

 Since $V'$ is a faithful representation of type
$(n_{1,1}-1,1,s_2-1)$ of $\mathfrak{m}$,
Proposition \ref{prop.type} implies
\[
(n_{1,1}-1)(s_2-1) \geq 6\qquad\text{ and }\qquad
n_{1,1}-1, s_2 - 1 \geq 3.
\]
Thus, the only possibility is $n_{1,1}= s_2 =4$.
That is, both $B_1$ and $B_3$, have 3 elements,
that is
\[
B= \{A_1v_1,A_2v_1,A_3v_1\}
  \cup \{A_4v_1\}  \cup \{v_{2},v_{3},v_{4}\}\cup\{v_1\}.
\]

Let $\{X_1, \dots,X_5\}$ be any set of generators of $\mathfrak{n}_1$
with $\{X_1, \dots, X_4\}\subset\mathfrak{m}$ and $X_5=A_{4}$.
If we denote by $\tilde X$ the matrix corresponding to the action of $X$ on $V$
associated to the basis $B$,
we know that
\[
 \tilde X_1,\dots,\tilde X_4\!=\!
 \left(\tiny
 \begin{array}{ccc|c|ccc|c}
  0 & 0 & 0 &  *  & * & * & * &   0 \\
  0 & 0 & 0 &  *  & * & * & * &   0 \\
  0 & 0 & 0 &  *  & * & * & * &   0 \\
  \hline
  0 & 0 & 0 &  0  & * & * & * &   0 \\
  \hline
  0 & 0 & 0 &  0  & 0 & 0 & 0 &   0 \\
  0 & 0 & 0 &  0  & 0 & 0 & 0 &   0 \\
  0 & 0 & 0 &  0  & 0 & 0 & 0 &   0 \\
  \hline
  0 & 0 & 0 &  0  & 0 & 0 & 0 &   0 \\
 \end{array}
 \right)\!,\;
 \tilde X_5\!=\!
  \left(\tiny
 \begin{array}{ccc|c|ccc|c}
  * & * & * &  *  & * & * & * &   0 \\
  * & * & * &  *  & * & * & * &   0 \\
  * & * & * &  *  & * & * & * &   0 \\
  \hline
  * & * & * &  *  & * & * & * &   1 \\
  \hline
  * & * & * &  *  & * & * & * &   0 \\
  * & * & * &  *  & * & * & * &   0 \\
  * & * & * &  *  & * & * & * &   0 \\
  \hline
  * & * & * &  *  & * & * & * &   0 \\
 \end{array}
 \right).
\]
By definition of $B_1$, we know that
$\text{span}_{\k}\{X_jA_4v_1:j=1,\dots,4\}=\text{span}_{\k} (B_1)$.
Therefore, after a Gaussian elimination process, we
can redefine $\{X_1, \dots, X_5\}$ so that (we do not change
$B_1$, $B_2$, but we may need to permute $B_3$)
\begin{equation}\label{eq.e_i}
 \tilde X_i =
 \left(\tiny
 \begin{array}{ccc|c|ccc|c}
  0 & 0 & 0 &    & * & * & * &   0 \\
  0 & 0 & 0 &  e_i  & * & * & * &   0 \\
  0 & 0 & 0 &    & * & * & * &   0 \\
  \hline
  0 & 0 & 0 &  0  & 0 & * & * &   0 \\
  \hline
  0 & 0 & 0 &  0  & 0 & 0 & 0 &   0 \\
  0 & 0 & 0 &  0  & 0 & 0 & 0 &   0 \\
  0 & 0 & 0 &  0  & 0 & 0 & 0 &   0 \\
  \hline
  0 & 0 & 0 &  0  & 0 & 0 & 0 &   0 \\
 \end{array}
 \right)\!,\; i=1,2,3,
\end{equation}
(here $\{e_1,e_2,e_3\}$ are the canonical vectors of $\k^3$) and
\[
\tilde X_4 =
 \left(\tiny
 \begin{array}{ccc|c|ccc|c}
  0 & 0 & 0 &  0  & * & * & * &   0 \\
  0 & 0 & 0 &  0  & * & * & * &   0 \\
  0 & 0 & 0 &  0  & * & * & * &   0 \\
  \hline
  0 & 0 & 0 &  0  & 1 & * & * &   0 \\
  \hline
  0 & 0 & 0 &  0  & 0 & 0 & 0 &   0 \\
  0 & 0 & 0 &  0  & 0 & 0 & 0 &   0 \\
  0 & 0 & 0 &  0  & 0 & 0 & 0 &   0 \\
  \hline
  0 & 0 & 0 &  0  & 0 & 0 & 0 &   0 \\
 \end{array}
 \right)\!,\;
 \tilde X_5\!=\!
 \left(\tiny
 \begin{array}{ccc|c|ccc|c}
  * & * & * &  0  & * & * & * &   0 \\
  * & * & * &  0  & * & * & * &   0 \\
  * & * & * &  0  & * & * & * &   0 \\
  \hline
  * & * & * &  *  & 0 & * & * &   1 \\
  \hline
  * & * & * &  *  & * & * & * &   0 \\
  * & * & * &  *  & * & * & * &   0 \\
  * & * & * &  *  & * & * & * &   0 \\
  \hline
  * & * & * &  *  & * & * & * &   * \\
 \end{array}
 \right).
\]
Moreover:
\begin{enumerate}
 \item Replacing $X_k$  by $X_k+t_1[X_1,X_4]+t_2[X_2,X_4]+ t_3[X_3, X_4]$, for some appropriate $t_1,t_2,t_3\in\k$, we can assume, for all $k$, that
 \[
  (\tilde X_k)_{1,5}=  (\tilde X_k)_{2,5}=  (\tilde X_k)_{3,5}=0.
 \]
\item Replacing $v_3$ by $v_3+t_1v_2$ and $v_4$ by $v_4+t_2v_2$
for some appropriate $t_1,t_2\in\k$, we can assume, without
changing the properties already obtained, that
 \[
  (\tilde X_4)_{4,6}=  (\tilde X_4)_{4,7}= 0.
 \]

 \item Since $\{[X_1,X_2],[X_1,X_3],[X_2,X_3]\}$ is linearly independent, it is necessary that
 these three 2-coordinates vectors
 \[
  \big( (\tilde X_1)_{4,6},  (\tilde X_1)_{4,7} \big),\quad
  \big( (\tilde X_2)_{4,6},  (\tilde X_2)_{4,7} \big),\quad
  \big( (\tilde X_3)_{4,6},  (\tilde X_3)_{4,7} \big)
 \]
 span a 2-dimensional space.
 We may assume
 that the first two vectors do
 (this may require to permute $B_1$ in order to
 keep property \eqref{eq.e_i}).
 In this case, replacing $v_3$ by
 $(\tilde X_1)_{4,6}v_3+(\tilde X_2)_{4,6}v_4$ and
 $v_4$ by
  $(\tilde X_1)_{4,7}v_3+(\tilde X_1)_{4,7}v_4$,
  we may assume
\[
  \big( (\tilde X_1)_{4,6},  (\tilde X_1)_{4,7} \big)=(1,0)
  \quad\text{and}\quad
  \big( (\tilde X_2)_{4,6},  (\tilde X_2)_{4,7} \big)=(0,1).
 \]

\item
Since  $\tilde X_5$ is nilpotent, the equation
$[[\tilde X_k,\tilde X_4],\tilde X_5]=0$ for $k= 1, 2, 3$
implies that the first three columns
and the 5th row of $\tilde X_5$ are zero.
Similarly the equation
$[[\tilde X_1,\tilde X_2],\tilde X_5]=0$
implies that the 6th and 7th rows of  $\tilde X_5$ are zero.

\item The equation
$[[\tilde X_1,\tilde X_5],\tilde X_1]=0$
implies $(\tilde X_5)_{4,4}=0$ and the
equation
$[[\tilde X_1,\tilde X_5],\tilde X_5]=0$
implies that the last row of $\tilde X_5$ is zero.
\end{enumerate}
At this point we have $[X_4,X_5]=0$, a contradiction, and thus
 $\dim V\ge 9$.
 The argument is similar for $r=4$ but easier.
\end{proof}

\begin{corollary}\label{TeoCotaInferior}
If $(\pi, V)$ is a faithful nilrepresentation of
$\mathcal{L}_{r,2}$ and
$r \geq 6$, then
\[
\dim V\ge \left\lceil \sqrt{2r(r-1)}  \right\rceil + 2.
\]
\end{corollary}

\begin{proof}
Let $v_1 \in V$ as in Theorem \ref{teo:descomposicion} and let
$\phi$ be the linear map
\begin{eqnarray}
\phi : \mathcal{L}_{r,2} \rightarrow V, \quad \phi(X)= \pi(X)(v_1). \nonumber
\end{eqnarray}
It is easy to check that
\begin{eqnarray}
\nonumber \ker  \phi&=& \mathfrak{n}_{1,2} \oplus \dots \oplus \mathfrak{n}_{1,s_1} \\
\nonumber \ker \phi\mid_{\mathfrak{z}(\mathcal{L}_{r,2})}&=& \mathfrak{n}_{2,2} \oplus \dots \oplus \mathfrak{n}_{2,s_2} \\
\label{eq:im} \im \phi&=& \mathfrak{n}_{1,1}v_1
\end{eqnarray}
It follows from Theorem \ref{teo:descomposicion}, part \eqref{it.4} and \eqref{eq:im}
$$
\dim V \geq \dim \im \phi + s_2.
$$
This implies
\[
\dim V + \dim \ker \phi \geq \dim \mathcal{L}_{r,2} + s_2
\]
and hence
\begin{equation}\label{eq.12}
\dim V + \dim(\mathfrak{n}_{1,2} \oplus \dots \oplus \mathfrak{n}_{1,s_1}) \geq  \dim(\mathfrak{n}_{2,1} \oplus\mathfrak{n}_{2,2} \oplus \dots \oplus \mathfrak{n}_{2,s_2}) + r + s_2.
\end{equation}
On the other hand, it follows from Proposition \ref{teoestruc} that
\begin{equation}\label{eq.13}
\dim(\mathfrak{n}_{1,2} \oplus \dots \oplus \mathfrak{n}_{1,s_1})\leq \dim(\mathfrak{n}_{2,2} \oplus \dots \oplus \mathfrak{n}_{2,s_2}) + r-1.
\end{equation}

We now consider two cases.
\begin{enumerate}[(A)]
\item If $\dim(\mathfrak{n}_{1,2} \oplus \dots \oplus \mathfrak{n}_{1,s_1}) < \dim(\mathfrak{n}_{2,2} \oplus \dots \oplus \mathfrak{n}_{2,s_2}) + r-1$, then it follows from \eqref{eq.12} and \eqref{eq.13} that
\begin{equation}\label{eq.11}
  \dim V\geq  \dim\mathfrak{n}_{2,1} + s_2 + 2.
\end{equation}
 Since
 \[
  \pi(\mathfrak{z}(\mathcal{L}_{r,2}))
  =\mathfrak{n}_{2}= \bigoplus_{j=1}^{s_2}\mathfrak{n}_{2,j}
 \]
 and
 $\dim \mathfrak{n}_{2,j}\le \dim \mathfrak{n}_{2,1}$ (see Remark
 \ref{rmk.dimension}) we have
 \[
  s_2 \dim \mathfrak{n}_{2,1} \ge \frac{r(r-1)}{2}
 \]
and hence $\dim \mathfrak{n}_{2,1}+ s_2\ge
\left\lceil 2\sqrt{\frac{r(r-1)}{2}} \right\rceil$.
This, combined with \eqref{eq.11}, implies
\[
\dim V\geq
\left\lceil \sqrt{2r(r-1)}  \right\rceil + 2.
\]

\item If $\dim(\mathfrak{n}_{1,2} \oplus \dots \oplus \mathfrak{n}_{1,s_1}) = \dim(\mathfrak{n}_{2,2} \oplus \dots \oplus \mathfrak{n}_{2,s_2}) + r-1$, from Proposition \ref{teoestruc}
\[
 \dim V\ge 2r-2.
\]
Hence, if $r \geq 6$ we obtain
\[
\dim V  \geq \left\lceil \sqrt{2r(r-1)}  \right\rceil + 2.
\]
\end{enumerate}
This completes the proof.
\end{proof}


\section{The Upper Bound for \texorpdfstring{$\mu(\mathcal{L}_{r,2})$}{m(Lr)}}\label{sec:upper}


Let $n\in\N$ a fixed natural number. It is not difficult to see that the natural numbers
 $a\ge b$ defined by
 \begin{equation}\label{eq:def_ab}
    a=\left\lceil\sqrt{n}\right\rceil\text{ and }
  b=\begin{cases}
     a-1, &\text{ if $a(a-1)\ge n$;} \\
     a, &\text{ if $a(a-1)< n$;}
    \end{cases}
 \end{equation}
 satisfy
  \begin{equation}\label{eq:condition_ab}
  \begin{split}
  ab &\ge n \\
  a+b& = \left\lceil2\sqrt{n}\right\rceil \\
     &=\min\{c+d:c,d\in \N\text{ and } cd\ge n\}.
 \end{split}
 \end{equation}
 We point out that $a,b$ might not be the only pair satisfying \eqref{eq:condition_ab}, for instance
if $n=26$, then $a,b=6,5$ but $a',b'=7,4$ also work.

\begin{definition}\label{def:squareroot}
Given $n\in\N$, we will say that
\emph{the integer square roots of $n$} are the numbers
$a\ge b$ defined in \eqref{eq:def_ab}.
\end{definition}

This section is devoted to prove the following theorem.

\begin{theorem}\label{teo rep fiel}
Let $r\in\mathbb{N}$ such that $r\ge 2$, and let $a\ge b$ be
the integer square roots of $\binom{r}{2}$.
Then there exists a faithful nilrepresentation
of $\mathcal{L}_{r,2}$ of type $(a,2,b)$. In particular
\[
\mu(\mathcal{L}_{r,2}) \leq a+b+2= \left\lceil \sqrt{2r(r-1)}  \right\rceil + 2.
\]
\end{theorem}

The main idea to prove the above theorem is to show that,
if $a\ge b$ are the integer square roots of $\binom{r}{2}$,
then any ``generic'' assignment

\begin{equation*}
  \setlength{\unitlength}{11pt}
\begin{picture}(14,10)(2,-1)
\thicklines
\put(2.5,4){$X_i\mapsto $}
\put(18,4){ for $i=1,\dots, r$.}
\put(5.5,0){\line(0,1){8}}
\put(15.5,0){\line(0,1){8}}
\put(15.3,8){\line(1,0){.2}}
\put(5.5,8){\line(1,0){.2}}
\put(5.5,0){\line(1,0){.2}}
\put(15.3,0){\line(1,0){.2}}

\linethickness{0.1mm}
\multiput(11.5,0)(0,.2){40}{\line(0,1){.1}}
\multiput(9.5,0)(0,.2){40}{\line(0,1){.1}}
\multiput(5.5,3)(.2,0){50}{\line(1,0){.1}}
\multiput(5.5,5)(.2,0){50}{\line(1,0){.1}}

\put(13.2,3.7){$*$}
\put(13.2,6.35){$0$}
\put(10.2,6.35){$*$}
\put(7.2,6.35){$ 0 $}
\put(7.2,3.7){$0$}
\put(10.2,3.7){$0$}
\put(7.2,1.2){$ 0 $}
\put(10.2,1.2){$0$}
\put(13.2,1.2){$0$}
\put(11.65,0){$\underbrace{\rule{40pt}{0pt}}_\text{\small{$b$}}$}
\put(9.5,0){$\underbrace{\rule{15pt}{0pt}}_\text{\small{$2$}}$}
\put(5.5,0){$\underbrace{\rule{40pt}{0pt}}_\text{\small{$a$}}$}

\put(15.35,6.25){$\left.\rule{0mm}{7mm}\right\}\text{{$a$}}$}
\put(15.35,3.75){$\left.\rule{0mm}{4mm}\right\}\text{{$2$}}$}
\put(15.35,1.25){$\left.\rule{0mm}{7mm}\right\}\text{{$b$}}$}
\end{picture}
\end{equation*}
has the property that $[X_i,X_j]$, $1\le i<j\le r$ are mapped to
a linearly independent set, and thus it
produces a faithful representation of $\mathcal{L}_{r,2}$.

This lead us to introduce the sets $\mathcal S_{a,b}$ as follows:
given $a, b\in\N$, let
$n$ and $i_0$ be the unique non-negative integers satisfying
\[
 ab=\frac{n(n+1)}{2}+i_0\text{ with $0\le i_0\le n$}.
\]
(we informally say that $n$ and $i_0$ are the triangular representation of $ab$).
Let $\mathcal S_{a,b}\subset M_{a,2}^{n+1}\times M_{2,b}^{n}$ be the set of
 all of sequences of matrices
\begin{equation} \label{eq:seq_matrices_general}
\begin{split}
 A_1,\dots,A_{i_0-1},A_{i_0}, A_{i_0}',
 A_{i_0+1},\dots,A_{n}& \in M_{a,2} \\
 B_1,\dots,B_{n}& \in M_{2,b},
\end{split}
\end{equation}
such that the following products
\begin{equation}\label{eq:matrices_basis_general}
 \begin{matrix}
  A_1B_1  \\
  A_2B_1       & A_2B_2 \\
  \vdots       &  \vdots & \ddots \\
  A_{i_0}B_1   & A_{i_0}B_2 & \dots &  A_{i_0}B_{i_0} \\[1mm]
  A'_{i_0}B_1  & A'_{i_0}B_2 & \dots &  A'_{i_0}B_{i_0} \\[1mm]
  A_{i_0+1}B_1 & A_{i_0+1}B_2 & \dots &  A_{i_0+1}B_{i_0}&  A_{i_0+1}B_{i_0+1} \\
  \vdots       & \vdots       &  \dots & \vdots           &\vdots        & \ddots \\
  A_{n}B_1 & A_{n}B_2 & \dots &  A_{n}B_{i_0} &  A_{n}B_{i_0} & \dots &  A_{n}B_{n}
 \end{matrix}
\end{equation}
constitute a basis of $M_{a,b}$.
The question is whether $\mathcal S_{a,b}$ is not empty.
This is partially answered in the following theorem.

\begin{theorem}\label{thm:faithful_products}
Let $a, b\in\N$ and let
$n$ and $i_0$ be the triangular representation of $ab$.
Assume that $a$ and $b$ satisfy the following conditions:
\begin{enumerate}[(1)]
 \item $a=b$ or $a=b+1$ and
 \item $i_0\le b$ whenever $a=b$,
\end{enumerate}
then  $\mathcal S_{a,b}$ is a
non-empty  Zariski open of $M_{a,2}^{n+1}\times M_{2,b}^{n}$.
\end{theorem}

\begin{proof}
First we recall that the condition determining whether the set of matrices in
\eqref{eq:matrices_basis_general} is linearly independent
corresponds to showing that a given determinant is not zero.
Therefore, we only need to show that $\mathcal S_{a,b}\ne \emptyset$.

It is not difficult to see that
$\mathcal S_{a,b}\ne \emptyset$ when $a,b\le 4$.
Therefore, since $a \geq b$, we can assume $a\ge b\ge4$.

We start the proof by pointing out that
\begin{equation}\label{eq:start}
 b<n <2b.
\end{equation}
Indeed, since $ab=\frac{n(n+1)}{2}+i_0$ with $i_0<n+1$, we have
\[
(n+2)(n+1)> 2ab \ge  2b^2>(b+2)(b+1)\quad \text{ (for $b\ge 4$)}
\]
and hence $b<n$.
In addition, since $ab=\frac{n(n+1)}{2}+i_0$ with  $i_0\ge0$, we have
\[
 n(n+1)\le 2ab \le 2(b+1)b<  2(n+1)b
\]
and hence $n<2b$.

We now proceed by induction on $ab$, we must consider different cases.

\medskip

\noindent
\textbf{$\bullet$ Case $a=b+1$ and $2b-i_0<n$:}
Let
$$
\tilde a=b  \text{ and } \tilde b=a-2=b-1.
$$
We want to apply the induction hypothesis to $\tilde a$ and $\tilde b$. Thus we  need
to check properties (1) and (2): since $\tilde a=\tilde b+1$, condition (1) is satisfied and condition (2)
is vacuous.

In order to continue, we need the triangular representation
of $\tilde a\tilde b$.
Since $n$ and $i_0$ are the triangular representation of $ab$ it follows that
\[
 \tilde a \tilde b=b(a-2)
 =\frac{n(n+1)}{2}+i_0-2b
 =\frac{n(n-1)}{2}+n-(2b-i_0).
\]

Let us denote $i_1=2b-i_0$. It follows from  \eqref{eq:start}
$2b>n\ge i_0$, thus
$$
0<i_1;
$$ and in this case we have assumed  $i_1<n$.
Therefore  $\tilde n=n-1$ and $\tilde i_0=n-i_1$
are the triangular representation of  $\tilde a\tilde b$
since we have shown
\[
 0< \tilde i_0\le\tilde n.
\]

We are now in a position to apply the induction hypothesis
to  $\tilde a$ and $\tilde b$ and thus we obtain that
 $\mathcal S_{\tilde a,\tilde b}\ne \emptyset$.
 Therefore, we may choose sequences of matrices
\begin{equation*}
\begin{split}
 \tilde A_1,\dots,\tilde A_{ \tilde i_0-1},\tilde A_{ \tilde i_0}, \tilde A_{ \tilde i_0}',
 \tilde A_{ \tilde i_0+1},\dots,\tilde A_{ \tilde n}& \in M_{ \tilde a,2} \\
 \tilde B_1,\dots,\tilde B_{ \tilde n}& \in M_{2, \tilde b},
\end{split}
\end{equation*}
such that the following products
\begin{equation}\label{eq:young}
 \begin{matrix}
  \tilde A_1\tilde B_1  \\
  \tilde A_2\tilde B_1 & \tilde A_2\tilde B_2 \\
  \vdots &\vdots & \ddots \\
  \tilde A_{\tilde i_0}\tilde B_1 & \tilde A_{\tilde i_0}\tilde B_2 & \dots &  \tilde A_{\tilde i_0}\tilde B_{\tilde i_0} \\[1mm]
  \tilde A'_{ \tilde i_0}\tilde B_1 & \tilde A'_{ \tilde i_0}\tilde B_2 & \dots &  \tilde A'_{ \tilde i_0}\tilde B_{ \tilde i_0} \\[1mm]
  \tilde A_{ \tilde i_0+1}\tilde B_1 & \tilde A_{ \tilde i_0+1}\tilde B_2 & \dots &  \tilde A_{ \tilde i_0+1}\tilde B_{ \tilde i_0}&  \tilde A_{ \tilde i_0+1}\tilde B_{ \tilde i_0+1} \\
  \vdots & \vdots& \dots&\vdots &\vdots & \ddots \\
  \tilde A_{ \tilde n}\tilde B_1 & \tilde A_{ \tilde n}\tilde B_2 & \dots &  \tilde A_{ \tilde n}\tilde B_{ \tilde i_0} &  \tilde A_{ \tilde n}\tilde B_{ \tilde i_0} & \dots &  \tilde A_{ \tilde n}\tilde B_{ \tilde n}
 \end{matrix}
\end{equation}
constitute a basis of $M_{\tilde a,\tilde b}$.
Moreover, since
$\mathcal S_{\tilde a,\tilde b} \subset M_{\tilde a,2}^{\tilde n+1}\times M_{2,\tilde b}^{\tilde n}$ is Zariski open,
we may additionally require that
\begin{equation}\label{property**}
\begin{split}
\text{``any subset of $b$ elements of the set of all the columns of the} \\[-1mm]
\text{matrices $\tilde A_1,\dots,\tilde A_{ \tilde i_0}, \tilde A_{ \tilde i_0}',
 ,\dots,\tilde A_{ \tilde n}$ be linearly independent.''}
\end{split}
\end{equation}

Let
\[
 B_i=
 \begin{cases}
  \tilde A_{n-i}^t, & \text{if $i< i_1$;} \\[2mm]
  (\tilde A'_{\tilde{i}_0})^t, & \text{if $i= i_1$;}\\[2mm]
  \tilde A_{n+1-i}^t, & \text{if $i> i_1$.}
 \end{cases}
\]
Given $X\in M_{q_1-2,q_2}$, then $\widehat X\in M_{q_1,q_2}$ denotes the matrix
$X$ with two null rows added at the bottom. Let
\[
 A_i=
 \begin{cases}
  \widehat{\tilde B_{n-i}^t}, & \text{if $i< i_1$;} \\[2mm]
  \left(\begin{smallmatrix}
                    0 & \dots & 0 & 0 & 0 \\[1mm]
                    0 & \dots & 0 & 1 & 0
                   \end{smallmatrix}\right)^t
                   , & \text{if $i= i_1$;}\\[3mm]
  \widehat{\tilde B_{n+1-i}^t}, & \text{if $i>i_1$.}
 \end{cases}
\]
for $i=1,\dots,n$,
and let
$A'_{i_0}=\left(\begin{smallmatrix}
                    0 & \dots & 0 & 1 & 1 \\[1mm]
                    0 & \dots & 0 & 0 & 0 \\
            \end{smallmatrix}\right)^t \in M_{a,2}$.

Now, among all the products \eqref{eq:matrices_basis_general} we have that
\begin{equation*}
 \begin{matrix}
  A_1B_1  \\
  A_2B_1 & A_2B_2 \\
  \vdots & \vdots & \ddots \\
  A_{i_1-1}B_1 & A_{i_1-1}B_2 & \dots &  A_{i_1-1}B_{i_1-1} \\[1mm]
  A_{i_1+1}B_1 & A_{i_1+1}B_2 & \dots &  A_{i_1+1}B_{i_1-1}&  A_{i_1+1}B_{i_1}&  A_{i_1+1}B_{i_1+1} \\
  \vdots & \vdots& \dots& \vdots& \vdots& \vdots& \ddots \\
  A_{n}B_1 & A_{n}B_2 & \dots &  A_{n}B_{i_1-1} &  A_{n}B_{i_1} &  A_{n}B_{i_1+1} & \dots &  A_{n}B_{n}
 \end{matrix}
\end{equation*}
are linearly independent, as they are the widehat of the transpose
of the products in \eqref{eq:young},
and  each product has
its two last rows equal to zero.

On the other hand, if
$
B_i=\left(\begin{smallmatrix} v_i \\[1mm] w_i  \end{smallmatrix}\right)
$
we know that
the submatrix consisting of the two last rows of each of the following matrices
\begin{equation*}
 \begin{matrix}
  A_{i_1}B_1 & A_{i_1}B_2 & \dots &  A_{i_1}B_{i_1} \\[1mm]
  A'_{i_0}B_1 & A'_{i_0}B_2 & \dots &  A'_{i_0}B_{i_1}  & \dots &  A'_{i_0}B_{i_0}
 \end{matrix}
\end{equation*}
(here we have assumed $i_0\ge i_1$ but it may happen $i_1<i_0$)
are, respectively, equal to
\begin{equation*}
 \begin{matrix}
 \left(\begin{smallmatrix} v_1 \\[2mm] 0  \end{smallmatrix}\right)
   &    \left(\begin{smallmatrix} v_2 \\[2mm] 0  \end{smallmatrix}\right)
   & \dots
   &    \left(\begin{smallmatrix} v_{i_1} \\[1mm] 0  \end{smallmatrix}\right) \\[3mm]
  \left(\begin{smallmatrix} w_1 \\[2mm] w_1  \end{smallmatrix}\right)
   &    \left(\begin{smallmatrix} w_2 \\[2mm] w_2  \end{smallmatrix}\right)
   & \dots
   &    \left(\begin{smallmatrix} w_{i_1} \\[1mm] w_{i_1}  \end{smallmatrix}\right)
   & \dots
   &    \left(\begin{smallmatrix} w_{i_0} \\[1mm] w_{i_0}  \end{smallmatrix}\right)
 \end{matrix}
\end{equation*}
which are linearly independet by \eqref{property**}.
Therefore, the set of products \eqref{eq:matrices_basis_general}
is a basis of $M_{a,b}$ and this completes the induction step in this case.

\medskip

\noindent
\textbf{$\bullet$ Case $a=b+1$ and $2b-i_0\ge n$:}
In this case, from \eqref{eq:start} we obtain
$i_0< b$. Let
$$
\tilde a=b \text{ and }  \tilde{b}=a-1=b.
$$
We want to apply the induction hypothesis to $\tilde a$ and $\tilde b$ and thus we  need
to check properties (1) and (2).
Since $\tilde a=\tilde b$, condition (1) is satisfied and, in order to check  condition (2) we need the triangular representation
of $\tilde a\tilde b$.
Since $n$ and $i_0$ are the triangular representation of $ab$ it follows that
\[
 \tilde a \tilde b=b(a-1)
 =\frac{n(n+1)}{2}+i_0-b
 =\frac{n(n-1)}{2}+n-(b-i_0).
\]
If $i_1=b-i_0$, we obtain
$$
\tilde{n}=n-1 \text{ and } \tilde{i}_0=n-i_1.
$$
Since
\[
 0< \tilde i_0\le\tilde n,
\]
we have $\tilde i_0$ and $\tilde n$ are the triangular representation of  $\tilde a\tilde b$.

Moreover, since in this case  $2b-i_0\ge n$, it follows that
$\tilde i_0\le b=\tilde b$. This shows that condition (2) is also satisfied
and we can apply the induction hypothesis
to  $\tilde a$ and $\tilde b$.

Now the argument is analogous to that of the previous case.
The main difference is that, in this case,
if $X\in M_{q-1,2}$, then $\widehat X\in M_{q,2}$ denotes the matrix
$X$ with one (instead of two) null row added at the bottom.

Since $\emptyset\ne \mathcal S_{\tilde a,\tilde b}\subset M_{\tilde a,2}^{\tilde n+1}\times M_{2,\tilde b}^{\tilde n}$ is Zariski open we may choose sequences of matrices
\begin{equation*}
\begin{split}
 \tilde A_1,\dots,\tilde A_{ \tilde i_0-1},\tilde A_{ \tilde i_0}, \tilde A_{ \tilde i_0}',
 \tilde A_{ \tilde i_0+1},\dots,\tilde A_{ \tilde n}& \in M_{ \tilde a,2} \\
 \tilde B_1,\dots,\tilde B_{ \tilde n}& \in M_{2, \tilde b},
\end{split}
\end{equation*}
such that \eqref{property**} and such that the following products
\eqref{eq:young}
constitute a basis of $M_{\tilde a,\tilde b}$.
Let
\[
 B_i=
 \begin{cases}
  \tilde A_{n-i}^t, & \text{if $i< i_1$;} \\[2mm]
  (\tilde A'_{n-i_1})^t, & \text{if $i= i_1$;}\\[2mm]
  \tilde A_{n+1-i}^t, & \text{if $i> i_1$;}
 \end{cases}
\]
let
$A'_{i_0}=\left(\begin{smallmatrix}
                    0 & \dots & 0  & 1 \\[1mm]
                    0 & \dots & 0  & 0 \\
            \end{smallmatrix}\right)^t \in M_{a,2}$
(note the difference with the previous case)
and, for $i=1,\dots,n$, let
\[
 A_i=
 \begin{cases}
  \widehat{\tilde B_{n-i}^t}, & \text{if $i< i_1$;} \\[2mm]
  \left(\begin{smallmatrix}
                    0 & \dots & 0 & 0  \\[1mm]
                    0 & \dots & 0 & 1
                   \end{smallmatrix}\right)^t
                   , & \text{if $i= i_1$;}\\[3mm]
  \widehat{\tilde B_{n+1-i}^t}, & \text{if $i>i_1$.}
 \end{cases}
\]
Now, among all the products \eqref{eq:matrices_basis_general} we have that
\begin{equation*}
 \begin{matrix}
  A_1B_1  \\
  A_2B_1 & A_2B_2 \\
  \vdots & & \ddots \\
  A_{i_1-1}B_1 & A_{i_1-1}B_2 & \dots &  A_{i_1-1}B_{i_1-1} \\[1mm]
  A_{i_1+1}B_1 & A_{i_1+1}B_2 & \dots &  A_{i_1+1}B_{i_1-1}&  A_{i_1+1}B_{i_0}&  A_{i_1+1}B_{i_0+1} \\
  \vdots & & & & & & \ddots \\
  A_{n}B_1 & A_{n}B_2 & \dots &  A_{n}B_{i_1-1} &  A_{n}B_{i_1} &  A_{n}B_{i_1+1} & \dots &  A_{n}B_{n}
 \end{matrix}
\end{equation*}
are linearly independent (as in the previous case) and each product has
its last (instead of two last) row equal to zero.
Finally, if
$
\left(\begin{smallmatrix} v_i \\[1mm] w_i  \end{smallmatrix}\right)
$
is the submatrix consisting of the two last rows of $B_i$,
we know that  the last row  of  each of the following matrices
\begin{equation*}
 \begin{matrix}
  A_{i_1}B_1 & A_{i_1}B_2 & \dots &  A_{i_1}B_{i_1} \\[1mm]
  A'_{i_0}B_1 & A'_{i_0}B_2 & \dots &  A'_{i_0}B_{i_1}  & \dots &  A'_{i_0}B_{i_0}
 \end{matrix}
\end{equation*}
are, respectively, equal to
\begin{equation*}
 \begin{matrix}
 (v_1) & \dots & (v_{i_1}) \\[1mm]
 (w_1) & \dots & (w_{i_1})& \dots & (w_{i_0} )
 \end{matrix}
\end{equation*}
which are linearly independet by \eqref{property**}.
Therefore, the set of products \eqref{eq:matrices_basis_general}
is a basis of $M_{a,b}$ and this completes the induction step in this case.

\medskip

\noindent
\textbf{$\bullet$ Case $a=b$ and $i_0<b$:} Let
\begin{align*}
\tilde a&=b, \\ \tilde b&=a-1.
\end{align*}
Since $\tilde a=\tilde b+1$, condition (2) is vacuous and
we can apply the induction hypothesis.
As in the previous case, if
\[
 i_1=b-i_0,
\]
then  $\tilde n=n-1$ and $\tilde i_0=n-i_1$
are the triangular representation of  $\tilde a\tilde b$.
Now the argument is the same as in the previous case.

\medskip

\noindent
\textbf{$\bullet$ Case $a=b=i_0$:} Again, let
\begin{align*}
\tilde a&=b, \\ \tilde b&=a-1.
\end{align*}
Since $\tilde a=\tilde b+1$, condition (2) is vacuous and
we can apply the induction hypothesis.

In contrast to the previous case, now  $\tilde n=n$ and $\tilde i_0=0$
are the triangular representation of  $\tilde a\tilde b$.
If we take $i_1=b$, then the argument is the same as in the case
 $a=b+1$ and $2b-i_0\ge n$.
\end{proof}

\begin{theorem}\label{thm:faithful_Lie}
Let $r\in\N$, $r\ge 2$, and let $a\ge b$ be
the integer square roots of $\binom{r}{2}$.
Then there exist sequences of matrices
\begin{align*}
 X_1,\dots,X_{r}& \in M_{a,2} \\
 Y_1,\dots,Y_{r}& \in M_{2,b},
\end{align*}
such that the following $\binom{r}{2}$ matrices
$Z_{i,j}=X_iY_j-X_jY_i$ for $1 \leq j < i \leq r$, are linearly independent in $M_{a,b}$.
\end{theorem}

\begin{proof}
 Let $n=r-1$ and let
\begin{equation} \label{eq:seq_matrices}
\begin{split}
 A_1,\dots,A_{i_0-1},A_{i_0}, A_{i_0}',
 A_{i_0+1},\dots,A_{n}& \in M_{a,2} \\
 B_1,\dots,B_{n}& \in M_{2,b},
\end{split}
\end{equation}
be the sequences provided by Theorem \ref{thm:faithful_products}.
We ignore the matrix  $A_{i_0}'$ and we rename
the sequence $A_1,\dots,A_{n}$ as
$A_2,\dots,A_{n+1}$.
It follows from Theorem \ref{thm:faithful_products}
that the set of matrices
\begin{equation}\label{eq:matrices2}
 \begin{matrix}
  A_2B_1  \\
  A_3B_1 & A_3B_2 \\
  \vdots & \vdots & \ddots \\
  A_{n+1}B_1 & A_{n+1}B_2 & \dots &  A_{n+1}B_{n}
 \end{matrix}
\end{equation}
is linearly independent.
Now, we define $A_1=0$ and $B_r=0$ and let
\[
X_i=\epsilon^i A_i \;\text{ and }\; Y_i=B_i\quad \text{ for }i=1,\dots,r,
\]
for some $\epsilon\ne 0$  to be defined later.
For $1 \leq j < i \leq r$, we have
\[
 C_{i,j}=\epsilon^i(A_iB_j-\epsilon^{j-i}A_jB_i).
\]
Since $\epsilon\ne 0$,
the set $\{ C_{i,j},\; 1 \leq j < i \leq r\}$
is linearly independent if and only if
$\{ A_iB_j-\epsilon^{j-i}A_jB_i,\; 1 \leq j < i \leq r\}$
is linearly independent.
Since the set of matrices in \eqref{eq:matrices2}
is linearly independent and linear independence is a non-vanishing polynomial condition on an infinite field, it follows
that there exists $\epsilon\ne0$ such that
$\{ C_{i,j},\; 1 \leq j < i \leq r\}$
is linearly independent.
\end{proof}

\begin{proof}[Proof of Theorem \ref{teo rep fiel}]
Let $a\ge b$ be the integer square roots of $\binom{r}{2}$.
By Theorem \ref{thm:faithful_Lie}, there exist sequences of matrices
\begin{align*}
 X_1,\dots,X_{r}& \in M_{a,2} \\
 Y_1,\dots,Y_{r}& \in M_{2,b}.
\end{align*}
It is easy to check that the following $r$ matrices in $M_{a+b+2,a+b+2}$
\[
\setlength{\unitlength}{11pt}
\begin{picture}(19,8)
\thicklines
\put(5.5,0){\line(0,1){8}}
\put(15.5,0){\line(0,1){8}}
\put(15.3,8){\line(1,0){.2}}
\put(5.5,8){\line(1,0){.2}}
\put(5.5,0){\line(1,0){.2}}
\put(15.3,0){\line(1,0){.2}}

\linethickness{0.1mm}
\multiput(11.5,0)(0,.2){40}{\line(0,1){.1}}
\multiput(9.5,0)(0,.2){40}{\line(0,1){.1}}
\multiput(5.5,3)(.2,0){50}{\line(1,0){.1}}
\multiput(5.5,5)(.2,0){50}{\line(1,0){.1}}

\put(13.2,4){$Y_i$}
\put(13,6.35){$0$}
\put(7.2,1.5){$ 0 $}
\put(10,6.35){$X_i$}
\put(7.2,6.35){$ 0 $}
\put(7.2,4){$0$}
\put(10.2,4){$0$}
\put(10.2,1.5){$0$}
\put(13.5,1.5){$0$}
\put(11.65,0){$\underbrace{\rule{40pt}{0pt}}_\text{\small{$b$}}$}
\put(9.5,0){$\underbrace{\rule{15pt}{0pt}}_\text{\small{$2$}}$}
\put(5.5,0){$\underbrace{\rule{40pt}{0pt}}_\text{\small{$a$}}$}

\put(15.35,6.25){$\left.\rule{0mm}{7mm}\right\}\text{{$a$}}$}
\put(15.35,3.75){$\left.\rule{0mm}{4mm}\right\}\text{{$2$}}$}
\put(15.35,1.25){$\left.\rule{0mm}{7mm}\right\}\text{{$b$}}$}
\end{picture}
\]

\

\noindent
generates a Lie subalgebra isomorphic to $\mathcal{L}_{r,2}$.
\end{proof}




\begin{thebibliography}{BML}



\bibitem{A} L. Auslander,
\emph{On a problem of Philip Hall},
Ann. Math. 2nd series, Vol. \textbf{86}(1) (1967), 112--116.


\bibitem{Ba} D. W. Barnes,
\emph{Faithful Representations of Leibniz algebras},
Proc.of Am. Math. Soc., Vol. \textbf{141}(9), (2013), 2991--2995.

\bibitem{Be} Y. Benoist,
 \emph{Une Nilvariete Non Affine},
  J. Diff. Geom., Vol. \textbf{41}, (1995), 21--52.

\bibitem{Bi} G. Birkhoff,
 \emph{Representability of Lie algebras and Lie groups by matrices},
  Ann. of Math., Vol. \textbf{38}(2), (1937), 526--532.

\bibitem{Bu1} D. Burde,
\emph{On a refinement of Ado's Theorem},
Archiv.Math., Vol. \textbf{70}(2), (1998), 118--127.

\bibitem{Bu} D. Burde,
\emph{Left-symmetric algebras, or pre-Lie algebras in geometry and physics},
Central European J. of Math., Vol.\textbf{4}(3), (2006), 323--357.

\bibitem{BEdG} D. Burde, B. Eick, A. de Graaf,
 \emph{Computing faithful representations for nilpotent Lie algebras},
 J. of Alg., Vol. \textbf{322}(3), (2009), 602--612.

\bibitem{BM1} D. Burde, W. Moens,
 \emph{Minimal Faithful Representations of Reductive Lie Algebras},
 Archiv der Mathematik., Vol. \textbf{89}(6), (2007), 513--523.

\bibitem{BM2} D. Burde, W. Moens,
 \emph{Faithful Lie algebra modules and quotients of the universal enveloping algebra},
 J. of Alg., Vol. \textbf{325}(1), (2011), 440--460.

\bibitem{CR1} L. Cagliero, N. Rojas,
\emph{Faithful representation of minimal dimension of current Heisenberg Lie algebras},
Int. J. Math. Vol. \textbf{20}(11), (2009), 1347--1362.

\bibitem{CR2}
L. Cagliero,  N. Rojas,
\emph{A lower bound for faithful representations of nilpotent Lie algebras},
Linear and Multilinear Algebra, Vol. \textbf{63}(11), (2015), 2135--2150.

 \bibitem{dG} W. de Graaf,
 \emph{Constructing faithful matrix representations of Lie algebras},
Proceedings of the 1997 International Symposium on Symbolic and Algebraic Computation.

 \bibitem{dGN} W. de Graaf, W. Nickel,
 \emph{Constructing faithful representations of finitely-generated
torsion-free nilpotent groups}, J. Symbolic Comput., Vol. \textbf{33}(1),(2002), 31--41.

 \bibitem{GS} F. Grunewald and D. Segal
 \emph{Some General Algorithms. II: Nilpotent Groups},
Annals of Mathematics, Second Series, Vol. \textbf{112}(3), (1980),
585--617.


\bibitem{J} S. A. Jennings,
 \emph{The group ring of a class of infinite nilpotent groups},
 Canad. J. Math. \textbf{7}, (1955), 169--187.

\bibitem{KB}
Y-F. Kang, C-M. Bai,
\emph{Refinement of Ado's theorem in low dimensions and application in affine geomery},
Communications in Algebra, Vol. \textbf{36}(1), (2008), 82--93.

\bibitem{LC} W. Liu and M. Chen,
\emph{The minimal dimensions of faithful representations for Heisenberg Lie superalgebras},
J. of Geometry and Physics
Vol. \textbf{89}, (2015), 17--23.

\bibitem{LO}
E. Lo and G. Ostheimer,
\emph{A Practical Algorithm for Finding Matrix Representations for Polycyclic Groups},
Journal of Symbolic Computation, Vol. \textbf{28}, (1999), 339-360.

 \bibitem{Mi} J. Milnor,
\emph{On fundamental groups of complete affinely flat manifolds},
Adv. Math., Vol. \textbf{25}, (1977), 178--187.

 \bibitem{Ne}  Y. Neretin,
 \emph{A construction of finite-dimensional faithful representation of Lie algebra},
 Rend. Circ. Mat. Palermo (2) Suppl.,  Vol. \textbf{71}, (2003), 159--161.

  \bibitem{N}  W. Nickel,
  \emph{Matrix representations for torsion-free nilpotent groups by Deep Thought},
  J. of Algebra, Vol. \textbf{300}, (2006) 376--383.

 \bibitem{Re} B.E. Reed,
 \emph{Representations of solvable Lie algebras},
 Mich. Math. J., Vol. \textbf{16}, (1969), 227--233.

\bibitem{Ro1} N. Rojas,
 \emph{Minimal Faithful Representation of the Heisenberg Lie algebra with abelian factor},
J. of Lie Theory, Vol. \textbf{23}(4), (2013), 1105--1114.

\bibitem{Ro2} N. Rojas,
 \emph{Faithful Representations of Minimal Dimension of 6-dimensional nilpotent Lie algebras},
 J. Algebra Appl., Vol. \textbf{15}(10), (2016), 1650191(1)--1650191(19).

\bibitem{S} I. Schur,
 \emph{Zur Theorie vertauschbarer Matrizen},
 J. Reine Angew. Mathematik, Vol. \textbf{130}, (1905), 66--76.

\bibitem{W} B. A. Wehrfritz,
 \emph{Faithful linear representations of certain free nilpotent groups},
 Glasgow Math. J., Vol. \textbf{37}, (1995), 33--36.

 \bibitem{Z} P. Zusmanovich,
 \emph{Yet another proof of the Ado Theorem},
 J. of Lie Theory, Vol. \textbf{26}(3), (2016), 673--681.
\end{thebibliography}
\end{document}